\documentclass[11pt,a4paper,english]{article}
\usepackage{amssymb}
\usepackage{amsfonts,yfonts}
\usepackage[all]{xy}
\usepackage{amsfonts,amscd,amssymb,amsthm}
\usepackage{amsfonts,amssymb,amsmath,amsxtra,amscd,amsthm,eucal,graphicx,graphics}
\usepackage{tikz}
\usepackage{enumerate}
 \usepackage[T1]{fontenc}
\usepackage[utf8]{inputenc}
\usepackage{babel}
\newcommand{\overbar}[1]{\mkern 1.5mu\overline{\mkern-1.5mu#1\mkern-1.5mu}\mkern 1.5mu}
\newcommand{\cupdot}{\mathbin{\mathaccent\cdot\cup}}
\newcommand{\kchromfamily}[1]{\mathcal{G}_{#1}(n)}
\newcommand{\kchromfamilystar}{\mathcal{G}^{\ast}_{k}(n)}

\theoremstyle{definition}
\newtheorem{theorem}{Theorem}[section]

\newtheorem{lemma}[theorem]{Lemma}
\newtheorem{corollary}[theorem]{Corollary}

\newtheorem{conjecture}[theorem]{Conjecture}

 \title{New Bounds for Chromatic Polynomials \\ and Chromatic Roots}
    \author{Jason Brown \and Aysel Erey\footnote{Corresponding author, e-mail: aysel.erey@gmail.com}}
    \date{Department of Mathematics and Statistics\\ Dalhousie University \\ Halifax, Nova Scotia, Canada B3H 3J5 \\[\baselineskip] }

\begin{document}

\maketitle

\begin{abstract}
If $G$ is a $k$-chromatic graph of order $n$ then it is known that the chromatic polynomial of $G$, $\pi(G,x)$, is at most $x(x-1)\cdots (x-(k-1))x^{n-k} = (x)_{\downarrow k}x^{n-k}$ for every $x\in \mathbb{N}$. We improve here this bound by showing that  
\[ \pi(G,x) \leq (x)_{\downarrow k} (x-1)^{\Delta(G)-k+1}  x^{n-1-\Delta(G)}\] 
for every $x\in \mathbb{N},$
where $\Delta(G)$ is the maximum degree of $G$.
Secondly, we show that if $G$ is a connected $k$-chromatic graph of order $n$ where $k\geq 4$ then $\pi(G,x)$ is at most $(x)_{\downarrow k}(x-1)^{n-k}$ for every real $x\geq n-2+\left( {n \choose 2} -{k \choose 2}-n+k \right)^2$ (it had been previously conjectured that this inequality holds for all $x \geq k$). Finally, we provide an upper bound on the moduli of the chromatic roots that is an improvment over known bounds for dense graphs.
\end{abstract}


\section{Introduction}

Let $G$ be a graph with vertex set $V(G)$ and edge set $E(G)$ (the {\em order} and {\em size} of the graph are, respectively, $|V(G)|$ and $|E(G)$). For a nonnegative integer $x$, an \textit{$x$-colouring} of $G$ is a function $f:V(G)\rightarrow \{1,\dots , x\}$ such that $f(u)\neq f(v)$ for every $uv\in E(G)$. The \textit{chromatic number} $\chi(G)$ is smallest $x$ for which $G$ has an $x$-colouring. We say that $G$ is \textit{k-chromatic} if $\chi(G)=k$. The well known \textit{chromatic polynomial} $\pi(G,x)$ is the polynomial whose values at nonnegative integral values of $x$ counts the number of $x$-colourings of $G$. The fact that $\pi(G,x)$ is a polynomial in $x$ follows from the well-known {\em edge addition - contraction formula}: $$\pi(G,x)=\pi(G+uv,x)+\pi(G\cdot uv,x)$$ if $u$ and $v$ are nonadjacent vertices of $G$. An \textit{$i$--colour partition} of $G$ is a partition of the vertices of $G$ into $i$ nonempty independent sets. Let $a_i(G)$ denote the number of $i$-colour partitions of $G$. It is easy to see that 
$$\pi(G,x)=\sum_{i=\chi(G)}^{n}a_i(G)\,(x)_{\downarrow i}$$
where $(x)_{\downarrow i}=x(x-1)\dots (x-i+1)$ is the $i$th \textit{falling factorial} of $x$ and $n$ is the order of $G$. Moreover, $a_{i}(G)$ also satisfies an edge addition - contraction formula, namely, $a_i(G)=a_i(G+uv)+a_i(G\cdot uv)$. We refer the reader to \cite{dongbook} for a general discussion of graph colourings and chromatic polynomials.

Let $\mathcal{G}_k(n)$ be the family of all $k$-chromatic graphs of order $n$. Given a natural number $x\geq k$, it is natural to inquire about the maximum number of $x$-colourings among $k$-chromatic graphs of order $n$, that is, among graphs in $\mathcal{G}_k(n)$. Tomescu~\cite{tomescuintrobook} studied this problem and showed the following:

\begin{theorem}\cite[pg.\,239]{tomescuintrobook}\label{tomesdisconn}
Let $G$ be a graph in $\mathcal{G}_k(n)$. Then for every $x\in \mathbb{N}$,
$$\pi(G,x)\leq (x)_{\downarrow k} \, x^{n-k}.$$
Moreover, when $x\geq k$, the equality is achieved if and only if $G\cong K_k\cupdot (n-k)K_1$ (the graph consisting of a $k$-clique plus $n-k$ isolated vertices).
\end{theorem}

The next natural problem is to maximize the number of $x$-colourings of a graph over the family of {\em connected} $k$-chromatic graphs of order $n$ (we denote this family by $\mathcal{C}_k(n)$).
Interestingly, the problem becomes much more complicated when the connectedness condition is imposed. The answer is trivial when $x=k = 2$, as any $2$-chromatic connected graph has precisely two $2$-colourings. It is well known that (see, for example, \cite{dongbook}) if $G$ is a connected graph of order $n$ then $\pi(G,x)\leq x(x-1)^{n-1}$ for every $x\in \mathbb{N}$ and furthermore, when $x\geq 3$ the equality is achieved if and only if $G$ is a tree. Therefore, for $k=2$ and $x\geq 3$, the maximum number of $x$-colourings of a graph in $\mathcal{C}_2(n)$ is equal to $x(x-1)^{n-1}$ and extremal graphs are trees.
 
Tomescu settled the problem for $x=k=3$ in \cite{tomescufrench2} and later extended it for $x\geq k=3$ in \cite{tomescu}  by showing that 
 if $G$ is a graph in $\kchromfamily{3}$ then
$$\pi(G,x)\leq (x-1)^{n}-(x-1) \ \ \ \text{for odd} \ n$$
and
$$\pi(G,x)\leq (x-1)^{n}-(x-1)^2 \ \ \ \text{for even} \ n$$
for every integer $x\geq 3$ and furthermore the extremal graph is the odd cycle $C_{n}$ when $n$ is odd and odd cycle with a vertex of degree $1$ attached to the cycle (denoted $C_{n-1}^1$) when $n$ is even.

One might subsequently think that maximizing the number of $x$-colourings of a graph in $\mathcal{C}_k(n)$ should depend on the value of $k$. Let $\mathcal{C}^*_k(n)$  be the set of all graphs in $\mathcal{C}_k(n)$ which have size $ {k\choose 2}+n-k$ and clique number $k$ (that is,   $\mathcal{C}^*_k(n)$  consists of graphs which are obtained from a $k$-clique by recursively attaching leaves). In \cite{tomescufrench} Tomescu considered the problem for $x=k\geq 4$ and  conjectured the following (see also \cite{tomescuextrbanach,tomescu}):

\begin{conjecture}\cite{tomescufrench} \label{conjecture}Let $G$ be a graph in  $\mathcal{C}_k(n)$  where $k\geq 4$. Then 
\[ \pi(G,k)\leq k!\, (k-1)^{n-k},\]
or, equivalently, $a_k(G)\leq (k-1)^{n-k},$
with the extremal graphs belong to  $\mathcal{C}^*_k(n)$ .
\end{conjecture}

The authors in \cite{dongbook} mention the following conjecture which broadly extends Conjecture~\ref{conjecture} to all nonnegative integers $x$:

\begin{conjecture}\cite[pg. 315]{dongbook} \label{tomesdongconj} Let $G$ be a graph in  $\mathcal{C}_k(n)$  where $k\geq 4$. Then for every $x\in \mathbb{N}$,
$$\pi(G,x)\leq (x)_{\downarrow k}(x-1)^{n-k}.$$
Moreover, for $x\geq k$, the equality holds if and only if $G$ belongs to  $\mathcal{C}^*_k(n)$.
\end{conjecture}

%
%

\vspace{0.25in}
It is not hard to see that Conjecture~\ref{tomesdongconj} implies Theorem~\ref{tomesdisconn} because the chromatic polynomial of a graph is equal to the product of chromatic polynomials of its connected components. However, the problem of maximizing the number of colourings appears more difficult when graphs are connected, since the answer to this problem depends on the value of $k$ (the structure of extremal graphs seem to be different for $k=2$ and $3$).
As Tomescu  points out  \cite{tomescuintrobook}, the difficulty may lie in the lack of a characterization of {\em $k$-critical} graphs (those minimal with respect to $k$-chromaticity) when $k\geq 4$. 

If $G \in \mathcal{C}^*_k(n) $ then $\pi(G,x)=(x)_{\downarrow k}\,(x-1)^{n-k}$ as one can first colour the clique of order $k$ and then recursively colour the remaining vertices (which have only one coloured neighbour). On the other hand, one can see that if $\pi(G,x)=(x)_{\downarrow k}\,(x-1)^{n-k}$ then $G\in  \mathcal{C}^*_k(n)$ because the multiplicity of the root $1$ of the chromatic polynomial of a graph $G$ is equal to the number of blocks of $G$ \cite[pg.\,35]{dongbook}. Therefore, in Conjecture~\ref{tomesdongconj}, the extremal graphs are automatically determined if one can show that $\pi(G,x)\leq (x)_{\downarrow k}\,(x-1)^{n-k}$.

In this article, we first improve Tomescu's general upper bound (Theorem~\ref{tomesdisconn}), and show that if $G \in \mathcal{G}_k(n)$, then 
\[ \pi(G,x)\leq (x)_{\downarrow k} (x-1)^{\Delta(G)-k+1}  x^{n-1-\Delta(G)}\]
for every $x\in \mathbb{N}$ (Theorem~\ref{maintheorem}). Secondly, we discuss Conjecture~\ref{tomesdongconj} and show that if $G\in \mathcal{C}_k(n)$ where $k\geq 4$ then $\pi(G,x)$ is at most $(x)_{\downarrow k}(x-1)^{n-k}$ for every real $x\geq n-2+\left( {n \choose 2} -{k \choose 2}-n+k \right)^2$ (Theorem~\ref{thmforxlargeenough}). Finally, we also give a new upper bound on the moduli of the chromatic roots of a graph (Theorem~\ref{chromaticboundaysel}); our bound improves previously known bounds for dense graphs.

\section{Main Results}

\subsection{An improved upper bound for  the number of $x$-colourings}

Our goal is to improve Theorem~\ref{tomesdisconn} by finding an upper bound that is dependent on the maximum degree in the graph. We start by considering the case where there is a {\em universal vertex}, that is one with degree $n-1$.

\begin{lemma}\label{universallemma}
Let $G$ be a graph in $\kchromfamily{k}$ having $\Delta(G)=n-1$. Then for every $x\in \mathbb{N}$,
$$\pi(G,x)\, \leq \, (x)_{\downarrow k} \, (x-1)^{n-k}.$$ Moreover for $x\geq k$, the equality holds if and only if $G\in \mathcal{C}^*_k(n)$.
\end{lemma}

\begin{proof}
Let $u$ be a vertex of $G$ with maximum degree. Since $u$ is a universal vertex, it cannot be in the same colour class with any other vertex. Therefore, $\chi(G-u)=k-1$ and $\pi(G,x)=x\cdot \pi(G-u, x-1)$. Now, by Theorem~\ref{tomesdisconn}, 
$$\pi(G-u,x)\leq (x)_{\downarrow k-1}\,  x^{(n-1)-(k-1)}$$
for every $x\in \mathbb{N}$ and equality holds for $x\geq k-1$ if and only if $G-u\cong K_{k-1}\cupdot (n-k)K_1$.
Replacing $x$ with $x-1$ in the latter inequality yields
$$\pi(G-u,x-1)\leq (x-1)_{\downarrow k-1}\,  (x-1)^{n-k}$$
for every integer $x\geq 1$ and equality holds for $x\geq k$ if and only if $G-u\cong K_{k-1}\cupdot (n-k)K_1$.
Hence, the result follows as $\pi(G,x)=x\cdot \pi(G-u, x-1)$ and $(x)_{\downarrow k}=x\,(x-1)_{\downarrow k-1}$.\\
\end{proof}

\begin{theorem}\label{maintheorem}
Let $G$ be a graph in $\kchromfamily{k}$. Then for every natural number $x$,
$$\pi(G,x)\, \leq \, (x)_{\downarrow k} \, (x-1)^{\Delta(G)-(k-1)} \, x^{n-1-\Delta(G)}.$$ 
\end{theorem}
\begin{proof}
We proceed by induction on the number of vertices. For the basis step, $n = k$ and $G$ is a complete graph, so $\pi(G,x)=(x)_{\downarrow k}$. Now the result is clear as $\Delta(K_k)=k-1$.

Now we may assume that $G$ is a $k$-chromatic graph of order $n \geq k+1$. If  $\Delta(G)=n-1$ then the result follows by Lemma~\ref{universallemma}. So let us assume that $\Delta(G)<n-1$. Let $u$ be a vertex of maximum degree.  Set $t=n-1-\Delta(G)$ and let $\{v_1,\dots , v_t\}$ be the set of non-neighbours of $u$ in $G$,\ (that is, $\{v_1,\dots , v_t\}=V(G)\setminus N_G[u]$).
We set $G_0 = G$ and
$$G_i=G_{i-1}+uv_i$$
$$H_i=G_i\cdot uv_i$$
 for $i=1,\dots , t$.  By repeated use of the edge addition-contraction formula,
$$\pi(G,x)=\pi(G_t,x)+\sum_{i=1}^t\pi(H_i,x).$$
It is clear that $k\leq \chi(G_t), \chi(H_i)\leq k+1$ for $i = 1,2,\ldots,t$. Also, observe that $G_t$ is a graph of order $n$ having $\Delta(G_t)=n-1$ and each $H_i$ is a graph of order $n-1$ having $\Delta(H_i)\geq \Delta(G)+i-1$, and hence 
\[ \Delta(H_i)-\Delta(G)-i+1\geq 0.\]

\noindent \underline{Claim 1:}\, $\pi(G_t,x)\leq(x)_{\downarrow k}\,(x-1)^{n-k}$  for every $x\in \mathbb{N}$.\\

\noindent \textit{Proof of Claim 1:} Since $\Delta(G_t)=n-1$, we obtain by Lemma~\ref{universallemma} that 
$$\pi(G_t,x)\, \leq \, (x)_{\downarrow \chi(G_t)} \, (x-1)^{n-\chi(G_t)}.$$
Also, $(x)_{\downarrow \chi(G_t)} \, (x-1)^{n-\chi(G_t)}\leq (x)_{\downarrow k} \, (x-1)^{n-k}$ as $\chi(G_t)\geq k$ and $(x)_{\downarrow k+1} \, (x-1)^{n-(k+1)} \leq (x)_{\downarrow k} \, (x-1)^{n-k}$ for $x \geq k$). Hence Claim 1 follows.\\

\noindent \underline{Claim 2:} \, $\pi(H_i,x) \leq (x)_{\downarrow k}\,(x-1)^{\Delta(G)+i-k}\, x^{n-i-\Delta(G)-1}$ for every $x\in \mathbb{N}$. \\

\noindent \textit{Proof of Claim 2:}  By the induction hypothesis on $H_i$, if $\chi(H_i)=k$ then
\begin{eqnarray*}
\pi(H_i,x) & \leq & (x)_{\downarrow k}\, (x-1)^{\Delta(H_i)-k+1} x^{n-2-\Delta(H_i)}
\end{eqnarray*}
and if $\chi(H_i)=k+1$ then 
\begin{eqnarray*}
\pi(H_i,x) & \leq & (x)_{\downarrow k+1}\, (x-1)^{\Delta(H_i)-(k+1)+1}\, x^{(n-1)-1-\Delta(H_i)}\\
              & = & (x)_{\downarrow k+1}(x-1)^{\Delta(H_i)-k} x^{n-2-\Delta(H_i)}\\
              & \leq & (x)_{\downarrow k}(x-1)^{\Delta(H_i)-(k-1)} x^{n-2-\Delta(H_i)}
              \end{eqnarray*}
for every $x\in \mathbb{N}$.

Since $\Delta(H_i)-\Delta(G)-i+1\geq 0$, we find that
$$(x-1)^{\Delta(H_i)-\Delta(G)-i+1}\leq x^{\Delta(H_i)-\Delta(G)-i+1},$$
which is equivalent to
$$(x-1)^{\Delta(H_i)-k+1}x^{n-2-\Delta(H_i)}\leq (x-1)^{\Delta(G)+i-k} x^{n-i-\Delta(G)-1} .$$
This completes the proof of Claim $2$.

\vspace{0.25in}

The inequality proven in Claim $2$ yields
\begin{eqnarray*}
\sum_{i=1}^t\pi(H_i,x) & \leq & \sum_{i=1}^t (x)_{\downarrow k}\,(x-1)^{\Delta(G)+i-k} \, x^{n-i-\Delta(G)-1}\\
&=&(x)_{\downarrow k}(x-1)^{\Delta(G)-k}x^{n-\Delta(G)-1}\sum_{i=1}^t\left(\frac{x-1}{x}\right)^i
\end{eqnarray*}

Summing the geometric series, we find
$$\sum_{i=1}^t\left(\frac{x-1}{x}\right)^i=\frac{1-\left(\frac{x-1}{x}\right)^{t+1}}{1-\left(\frac{x-1}{x}\right)}-1.$$
Now, simplifying the expression on the right hand side of the latter equality and then substituting $t=n-1-\Delta(G)$ we get
$$\sum_{i=1}^t\left(\frac{x-1}{x}\right)^i=(x-1)-\frac{(x-1)^{n-\Delta(G)}}{x^{n-1-\Delta(G)}}.$$
Therefore, 
\begin{eqnarray*}
\sum_{i=1}^t\pi(H_i,x) & \leq & (x)_{\downarrow k}(x-1)^{\Delta(G)-k}x^{n-\Delta(G)-1}\left( (x-1)-\frac{(x-1)^{n-\Delta(G)}}{x^{n-1-\Delta(G)}}\right)\\
& = & (x)_{\downarrow k}\left((x-1)^{\Delta(G)-k+1}x^{n-\Delta(G)-1}-(x-1)^{n-k}\right).
\end{eqnarray*}
Furthermore, recall that $\pi(G_t,x) \leq (x)_{\downarrow k}\,(x-1)^{n-k}$ by the inequality proven in Claim $1$, so  
\begin{eqnarray*}
\pi(G,x) & = & \pi(G_t,x)+\sum_{i=1}^t\pi(H_i,x)\\
           & \leq & (x)_{\downarrow k}\,(x-1)^{n-k} + (x)_{\downarrow k}\left((x-1)^{\Delta(G)-k+1}x^{n-\Delta(G)-1}-(x-1)^{n-k}\right)\\
 & = & (x)_{\downarrow k} \, (x-1)^{\Delta(G)-k+1} \, x^{n-\Delta(G)-1}
\end{eqnarray*}
and we are done.
\end{proof}

\subsection{Maximizing the number of colourings for connected  graphs of fixed order and chromatic number}

Conjecture~\ref{tomesdongconj} is true for many graph families. For example, Tomescu~\cite{tomescu} proved it for  $ k=4$  under the additional restriction of $G$ being also planar. Also, it is easy to see that if the clique number of graph $G$ in $\mathcal{C}_k(n)$ is equal to $k$ then $G$ contains a spanning subgraph which is isomorphic to a graph in $\mathcal{C}^*_k(n)$. Therefore, Conjecture~\ref{tomesdongconj} holds for every graph $G$ in  $\mathcal{C}_k(n)$ having $\omega(G)=k$ (such graphs include all {\em perfect} graphs \cite{westbook}).

It is known that (see, for example, \cite{tomescufrench, westbook}) the minimum number of edges of a graph in $\mathcal{C}_k(n)$ is equal to ${k \choose 2}+n-k$. Furthermore, when $k=3$, the extremal graphs are unicyclic graphs with an odd cycle, and when $k\neq 3$, extremal graphs belong to $\mathcal{C}^*_k(n)$. As chromatic polynomial of a graph of order $n$ with $m$ edges has the form $\pi(G,x) = x^{n} - m\,x^{n-1} + \cdots$ it is not difficult to see that Conjecture~\ref{tomesdongconj} holds for all sufficiently large $x$. However it becomes quite difficult to find the smallest such value of $x$.

We begin with a lemma which gives an upper bound for the number of colour partitions of a graph.

\begin{lemma}\label{boundpartitions}Let $G$ be a graph of order $n$ and size $m$. Then for $1\leq i \leq n-1$,
$$a_i(G)\leq \frac{1}{(n-i)!}\left( {n \choose 2} -m \right)^{n-i}.$$

\end{lemma}

\begin{proof}
We proceed by induction on ${n \choose 2}-m$, the number of non-edges of the graph. For the basis step, suppose that $G$ is a complete graph. Then, $a_i(G)=0$ for $1\leq i \leq n-1$ and $a_n(G)=1$. Hence the result is clear. Now we may assume that $G$ has at least one pair of nonadjacent vertices, say $u$ and $v$. The graph $G+uv$ has order $n$ and size $m+1$. Also, the graph $G\cdot uv$ has order $n-1$ and size $m-|N_G(u)\cap N_G(v)|$. Thus the number of non-edges of $G+uv$ and $G\cdot uv$ are strictly less than the number of non-edges of $G$. Note that if $i=n-1$ then the result is clear since $a_{n-1}(G)={n\choose 2}-m$, so we may assume that $1\leq i \leq n-2$. 
Set $\beta={n \choose 2}-m$. Then by the induction hypothesis,
$$a_i(G+uv)\leq \frac{1}{(n-i)!}(\beta -1)^{n-i}$$
and $$a_i(G\cdot uv)\leq \frac{1}{(n-1-i)!}(\beta -1)^{n-1-i}.$$
By the edge addition-contraction formula,
$$a_i(G)=a_i(G+uv)+a_i(G\cdot uv).$$ Therefore,
\begin{eqnarray*}
a_i(G) & \leq & \frac{1}{(n-i)!}(\beta -1)^{n-i} + \frac{1}{(n-1-i)!}(\beta -1)^{n-1-i}\\
& = & \frac{1}{(n-i)!}\left( (\beta -1)^{n-i} + (n-i) (\beta -1)^{n-1-i} \right)\\
& \leq & \frac{1}{(n-i)!} \sum_{j=0}^{n-i} {n-i \choose j}(\beta -1)^{n-i-j}\\
& = & \frac{1}{(n-i)!} \, \beta^{n-i}.\\
\end{eqnarray*}
Thus, the proof is complete.
\end{proof}

\vspace{0.25in}
Let $f(z)=\sum_{i=0}^d c_i z^i$ be a real polynomial of degree $d \geq 1$. Then the \textit{Cauchy bound of f} (see, for example, \cite[pg.\,243]{rahman}), denoted by $\rho(f)$, is defined as the unique positive root of the equation 
\begin{center}
$|c_0|+|c_1|x+\dots +|c_{d-1}|x^{d-1}=|c_d|x^d$
\end{center}
when $f$ is not a monomial, and zero otherwise (the fact that $f$ has a unique positive real root follows from the intermediate value theorem and Descartes' rule of signs). It is known that the maximum of the moduli of the roots of $f$  is bounded by $\rho(f)$, and  the Cauchy bound satisfies (see \cite[pg. 247]{rahman})
\begin{equation}\label{cauchybound}
\rho(f)\leq 2\, \operatorname{max} \left\lbrace \left| \frac{c_i}{c_d} \right|^{1/(d-i)} \right\rbrace_{0\leq i\leq d-1}.
\end{equation}

\vspace{0.25in}

Let $\xi_1,\xi_2,\dots$ be a sequence of real numbers. Then the polynomials
$$P_0(z):=1, \qquad P_d(z):=\prod_{j=1}^d(z-\xi_j) \qquad (d=1,2,\dots)\label{newton}$$
are called the \textit{Newton bases with respect to the nodes} $\xi_1,\xi_2,\dots$; they form a basis for the vector space of all real polynomials \cite[pg. 256]{rahman}. 

\begin{theorem}\cite[pg. 266]{rahman} \label{rahman}
Let $f(z)=\sum_{j=0}^dc_jP_j(z)$ be a polynomial of degree $d$ where $P_j$'s are the Newton bases with respect to the nodes $\xi_1,\dots \xi_d$. Then $f$ has all its roots in the union of the discs
$$\mathcal{D}_j:=\{z\in \mathbb{C} : |z-\xi_j|\leq \rho\} \qquad (j=1,\dots,d)$$ 
where $\rho$ is the Cauchy bound of $\sum_{j=0}^dc_jz^j$.
\end{theorem}

\begin{theorem}\label{thmforxlargeenough}
Let $G$ be a graph in $\mathcal{C}_k(n) \setminus \mathcal{C}^*_k(n)$ where $k\geq 4$. Then 
$$\frac{1}{(x)_{\downarrow k}} \pi(G,x)<  (x-1)^{n-k}$$
for every real number $x$ where $x> n-2+\left( {n \choose 2} -{k \choose 2} -n+k  \right)^2$.
\end{theorem}

\begin{proof} Let $G^*$ be a graph in $\mathcal{C}^*_k(n)$. Then $\pi(G^*,x)=(x)_{\downarrow k}(x-1)^{n-k}$. Let 
\begin{eqnarray*}
f(x) &=& \frac{1}{(x)_{\downarrow k}} \left( \pi(G^*,x)- \pi(G,x) \right)\\
 &=& \frac{1}{(x)_{\downarrow k}} \sum_{r=k}^n\left( a_r(G^*)- a_r(G) \right)(x)_{\downarrow r}.
\end{eqnarray*}
Now, $a_n(G)=a_n(G^*)=1$. Also, $a_{n-1}(G^*)={n \choose 2}-{k \choose 2}-(n-k)$ and $a_{n-1}(G)={n \choose 2}-m$. Since $m>{k \choose 2}+(n-k)$ we have   $a_{n-1}(G^*) >  a_{n-1}(G)$. Therefore, $f(x)$ is a polynomial of degree $n-k-1$ with the leading coefficient  $a_{n-1}(G^*)-a_{n-1}(G)>0$. As the leading coefficient of the polynomial $f$ is  positive, it suffices to show that the largest real root of $f$ is at most  $n-2+\left( {n \choose 2} -{k \choose 2} -n+k  \right)^2$. Indeed, we shall prove a stronger statement, namely that if $z\in \mathbb{C}$ is a root of $f$ then $\Re(z)\leq n-2+\left( {n \choose 2} -{k \choose 2} -n+k  \right)^2$.

Set $\alpha_r=a_r(G^*)-a_r(G)$. Thus $\alpha_{n-1}=a_{n-1}(G^*)-a_{n-1}(G) > 0 $ and all $\alpha_r$'s are integers, and 
\[ f(x) = \alpha_k+ \alpha_{k+1}(x-k)+ \alpha_{k+2}(x-k)(x-k-1)+\cdots +\alpha_{n-1}(x-k)\cdots (x-n+2)\]
that is, $$f(x)=\sum_{j=0}^{n-1-k}\alpha_{k+j}\,  P_j(x)$$ where $P_j(x)$'s are Newton bases with respect to nodes $k, \ k+1, \dots , \ n-2$.

By Theorem~\ref{rahman}, $f$ has all its roots in the union of the discs centered at $$k, \ k+1, \dots , n-3, \ n-2$$
each of radius $\rho$ where $\rho$ is the Cauchy bound of 
$$g=\alpha_{n-1}z^{n-k-1}+\alpha_{n-2}z^{n-k-2}+\alpha_{n-3}z^{n-k-3}+\cdots +\alpha_k.$$ By the inequality given in (\ref{cauchybound}), the Cauchy bound of $g$ satisfies
$$\rho\leq 2 \operatorname{max} \left\lbrace \left| \frac{\alpha_{n-r}}{\alpha_{n-1}} \right|^{1/(r-1)} \right\rbrace_{2\leq r\leq n-k}.$$
Note that as all of the $\alpha_{r}$'s are integers with $\alpha_{n-1} > 0 $,
$$ \left| \frac{\alpha_{n-r}}{\alpha_{n-1}} \right| \leq |\alpha_{n-r}| \leq \operatorname{max} \{a_{n-r}(G), a_{n-r}(G^*)\}.$$
Moreover, by Lemma~\ref{boundpartitions},
$$a_{n-r}(G)\leq  \frac{\left( {n \choose 2} -m  \right)^r}{r!} \, \, \, \,\,\,   \text{and} \, \,\, \, \, a_{n-r}(G^*)\leq  \frac{\left( {n \choose 2} -{k\choose 2}-n+k  \right)^r}{r!}.$$
Now, since $m> {k\choose 2}+n-k $ we obtain that
$$\operatorname{max} \{a_{n-r}(G), a_{n-r}(G^*)\} \leq  \frac{\left( {n \choose 2} -{k\choose 2}-n+k  \right)^r}{r!}.$$
So,
$$\left| \frac{\alpha_{n-r}}{\alpha_{n-1}} \right|^{1/(r-1)} \leq \left( \frac{\left( {n \choose 2} -{k \choose 2} -n+k  \right)^r}{r!} \right)^{1/(r-1)}=\frac{\left( {n \choose 2} -{k \choose 2} -n+k  \right)^{r/(r-1)}}{(r!)^{1/(r-1)}}.$$
As $r$ increases, $\left( {n \choose 2} -{k \choose 2} -n+k  \right)^{r/(r-1)}$ decreases and $(r!)^{1/(r-1)}$ increases. Hence \[ \left\lbrace \frac{\left( {n \choose 2} -{k \choose 2} -n+k  \right)^{r/(r-1)}}{(r!)^{1/(r-1)}} \right\rbrace_{2\leq r\leq n-k}\] is a decreasing sequence and therefore,
$$ \operatorname{max} \left\lbrace \left| \frac{\alpha_{n-r}}{\alpha_{n-1}} \right|^{1/(r-1)} \right\rbrace_{2\leq r\leq n-k} \leq \frac{\left( {n \choose 2} -{k \choose 2} -n+k  \right)^2}{2} .$$
Thus, we obtain that $\rho \leq \left( {n \choose 2} -{k \choose 2} -n+k  \right)^2$ and the result follows.
\end{proof}

\subsection{Chromatic Roots}

A {\em chromatic root} is a root of a chromatic polynomial. The chromatic number of $G$ is one more than the largest {\em integer} chromatic root of $G$. There has been considerable interest in chromatic roots,  particularly on bounding the moduli of the roots (see, for example, \cite[Ch. 14]{dongbook}.
In this section, by considering the complete graph expansion of the chromatic polynomial, we will give a new bound for the moduli of chromatic roots of all graphs. This bound is sharp and the equality is obtained when the graph is a complete graph.

We will need the following theorem that locates the roots of a polynomial expressed in terms of Newton bases.

\begin{theorem}\label{rahmannewtonbound}\cite[pg.\,267]{rahman}
Let $f(z)=\sum_{i=0}^d c_i P_i(z)$ be a polynomial of degree $d$ where $P_i(z)$'s are Newton bases with respect to the nodes $\xi_1 ,\dots, \xi_d$.  Denote by $\rho$ the Cauchy bound of $c_dz^d+\sum_{i=0}^{d-2}c_i z^i$. Then $f$ has all its roots in the union $\mathcal{U}$ of the discs centered at $\xi_1 ,\dots \xi_{d-1}, \xi_d-\frac{c_{d-1}}{c_d}$, each of radius $\rho$.
\end{theorem}

We are ready to prove our new bound on chromatic roots. 

\begin{theorem}\label{chromaticboundaysel}
Let $G$ be a $k$-chromatic graph of order $n$ and size $m$. Then $\pi(G,z)$ has all its roots in $\{0,1,\dots , k-1\}\cup \mathcal{U}$ where $\mathcal{U}$ is the union of the discs centered at 
$$k, \, k+1,\dots ,n-2, \, n-1-{n \choose 2}+m,$$
each of radius $\sqrt{2} \left( {n\choose 2} -m \right) $.
Thus the moduli of a chromatic root of a graph of order $n$ with $m$ edges is bounded above by $n-1 + \sqrt{2} \left( {n \choose 2} - m \right)$.
\end{theorem}

\begin{proof}
First recall that $\displaystyle \pi(G,z)=\sum_{i=k}^na_i(G)\,(z)_{\downarrow i}$. Therefore $\pi(G,z)=(z)_{\downarrow k}\, f(z)$  and the roots of $\pi(G,z)$ are precisely $\{0,1,\dots , k-1\}$ union the roots of $f(z)$. Hence, it suffices to show that the roots of $f(z)$ lie in $\mathcal{U}$. Now,

\begin{eqnarray*}
f(z) & = & a_k \, \frac{(z)_{\downarrow k}}{(z)_{\downarrow k}} + a_{k+1} \, \frac{(z)_{\downarrow k+1}}{(z)_{\downarrow k}} + \cdots + a_n \, \frac{(z)_{\downarrow n}}{(z)_{\downarrow k}} \\
& =& a_{k}+a_{k+1}(z-k)+\dots +a_n(z-k)\cdots (z-n+1).
\end{eqnarray*}

Hence,
$$f(z)=\sum_{j=0}^{n-k}a_{k+j}P_j(z)$$
where  $P_j(z)$'s are Newton bases with respect to the nodes $k,k+1,\dots ,n-1$. Therefore, by Theorem~\ref{rahmannewtonbound}, $f(z)$ has all its roots in the union of the discs centered at 
$$k, \, k+1,\dots ,n-2, \, n-1-{n \choose 2}+m$$
each of radius $\rho$ where $\rho$ is the Cauchy bound of the polynomial
$$g=a_n z^{n-k}+a_{n-2}z^{n-k-2}+\cdots +a_{k+1}z+a_k.$$
Since $a_n=1$, by the inequality given in (\ref{cauchybound}) we obtain
$$\rho(g) \leq 2 \operatorname{max} \left\lbrace a_{n-r}(G) ^{1/r} \right\rbrace_{2\leq r\leq n-k}.$$
Also, by Lemma~\ref{boundpartitions}, we get
$$a_{n-r}(G) ^{1/r}\leq \left( \frac{\left( {n\choose 2} -m \right)^r}{r!} \right)^{1/r}= \frac{ {n\choose 2} -m }{(r!) ^{1/r}}.$$
Now, $(r!)^{1/r}$ increases as $r$ increases. Therefore, $a_{n-r}(G)\leq \frac{1}{\sqrt{2}}\left( {n\choose 2} -m \right)$ for $2\leq r \leq n-k$. Thus, $\rho(g)\leq \sqrt{2}\left( {n\choose 2} -m \right)$ and the results follow.
\end{proof}

\begin{corollary}Let $G$ be a graph of order $n$ and size $m$. If $z$ is a root of  $\pi(G,z)$ then

$$|\Im(z)|\leq \sqrt{2}\left( {n \choose 2} - m \right),$$
$$\Re(z)\leq n-1+\sqrt{2}\left( {n \choose 2} - m \right).$$
\end{corollary}

\begin{table}
\begin{center}
\begin{tabular}{l||c|c|c|}

Graph $G$    &  Sokal bound & Fernandez-Procacci bound & New bound\\
& & & \\
\hline
\hline
& & & \\
$\overbar{G}$ is a tree      & $7.964(n-2)$   & $6.908(n-2)$ & $2.414(n-1)$  \\
& & & \\
$\overbar{G}$ is a cycle      & $7.964(n-3)$   & $6.908(n-3)$ & $2.414n-1$  \\
& & & \\
 $\overbar{G}$ is  a theta graph  & $7.964(n-3)$   & $6.908(n-3)$ & $2.414n+0.414$  \\ 
& & & \\
 $\overbar{G}$ is $3$-regular   & $7.964(n-4)$   & $6.908(n-4)$    & $3.121n-1$  \\
 & & & \\
  $\overbar{G}$ is $4$-regular   & $7.964(n-5)$   & $6.908(n-5)$    & $3.828n-1$  \\
\end{tabular}
\end{center}
\caption{Comparison of bounds for the chromatic roots of a graph $G$ of order $n$ and size $m$ whose complement $\overbar{G}$ is a cycle, tree, $3$-regular graph or theta graph.}
\label{boundtable}
\end{table}

Sokal \cite{sokal} proved the moduli of chromatic roots are bounded by $7.964 \Delta$, with an improvement in the constant to $6.908$ in \cite{fernandez}. 
Table~\ref{boundtable} compares our new bound on the moduli to these, for a variety of dense of graphs. Note the significant improvement in the constant in linear upper bounds. In particular, for any family of $r$-regular graphs with $r \geq n-8$, our bounds are asymptotically much better than the others.

\section{Concluding remarks}

In \cite{dongbook} it was shown that if $G$ is a connected graph of order $n$, then for every $x\in \mathbb{N}$,
$$\pi(G,x)~\leq ~x(x-1)^{n-1}$$
where equality holds for $x\geq 3$ if and only if $G$ is a tree. From this we can prove that to prove Conjecture~\ref{tomesdongconj}, it is sufficient to prove it for $2$-connected graphs.

\begin{lemma}\label{transitionlemma}Let $G$ be a graph in $\kchromfamily{k}$ consisting of $t$ blocks $B_1,\dots , B_t$  and $n_i$ be the order of $B_i$. Let also $x$ be a natural number. Suppose that for some block $B_i$ with $\chi(B_i)=k$, the inequality $\pi(B_i,x)\, \leq \, (x)_{\downarrow k}\,(x-1)^{n_i-k}$ holds. Then,
$$\pi(G,x)\, \leq \, (x)_{\downarrow k}\,(x-1)^{n-k}.$$
Moreover, for $x\geq k$ the equality $\pi(G,x)=(x)_{\downarrow k}\,(x-1)^{n-k}$ holds if and only if  $G$ has exactly one $k$-chromatic block, say $B_i$, and for this block the equality $\pi(B_i,x)\, = \, (x)_{\downarrow k}\,(x-1)^{n_i-k}$ holds, and all the rest of the blocks are $K_2$'s.
\end{lemma}
\begin{proof}
Clearly $n_1+n_2 + \cdots n_t = n+t-1$. Let $B_1$ be a block of $G$ such that $\chi(B_1)=k$ and $\pi(B_1,x)\, \leq \, (x)_{\downarrow k}\,(x-1)^{n_1-k}$. Since $B_i$ is a connected graph, $\pi(B_i,x)\leq x(x-1)^{n_i-k}$  for each $i\geq 2$, as noted earlier. Also, the {\em Complete Cutset Theorem} for chromatic polynomials (see, for example, \cite{dongbook})
states that if $G_1$ and $G_2$ be two graphs that overlap in a clique of size $r$, then $\displaystyle{\pi(G_1\cup G_2,x)=\frac{\pi(G_1,x)\, \pi(G_2,x)}{(x)_{\downarrow r}}}$. From this result, we derive that
\begin{eqnarray*}
\pi(G,x) & = & \pi(B_1,x)\, \frac{\pi(B_2,x)}{x}\dots \, \frac{\pi(B_t,x)}{x}\\
           & \leq & (x)_{\downarrow k}\,(x-1)^{n_1-k} (x-1)^{n_2-1} \cdots (x-1)^{n_t-1}\\
           & = & (x)_{\downarrow k}\,(x-1)^{n_1+n_2 + \cdots n_t-k-(t-1)}\\
           & = & (x)_{\downarrow k}\,(x-1)^{n+t-1-k-(t-1)}\\
           & = & (x)_{\downarrow k}\,(x-1)^{n-k}.
\end{eqnarray*}
Now, $\pi(G,x)\, = \, (x)_{\downarrow k}\,(x-1)^{n-k}$ if and only if $\pi(B_1,x)\, = \, (x)_{\downarrow k}\,(x-1)^{n_1-k}$ and $\pi(B_i,x)\, = \, x(x-1)^{n_i-1}$ for $i\geq 2.$ The latter equality holds if and only if $B_i$ is a tree. But since $B_i$ is a block this means that $B_i$ is equal to a $K_2$.
\end{proof}

In \cite{tomescu3chromblock}, the maximum number of $x$-colourings of a  $2$-connected $3$-chromatic graph of order $n$ was determined. For $k\geq 4$, from some computations on small graphs, we have noted that the following strengthening of Conjecture~\ref{tomesdongconj} might hold.

\begin{figure}
\begin{center}
\includegraphics[width=4in]{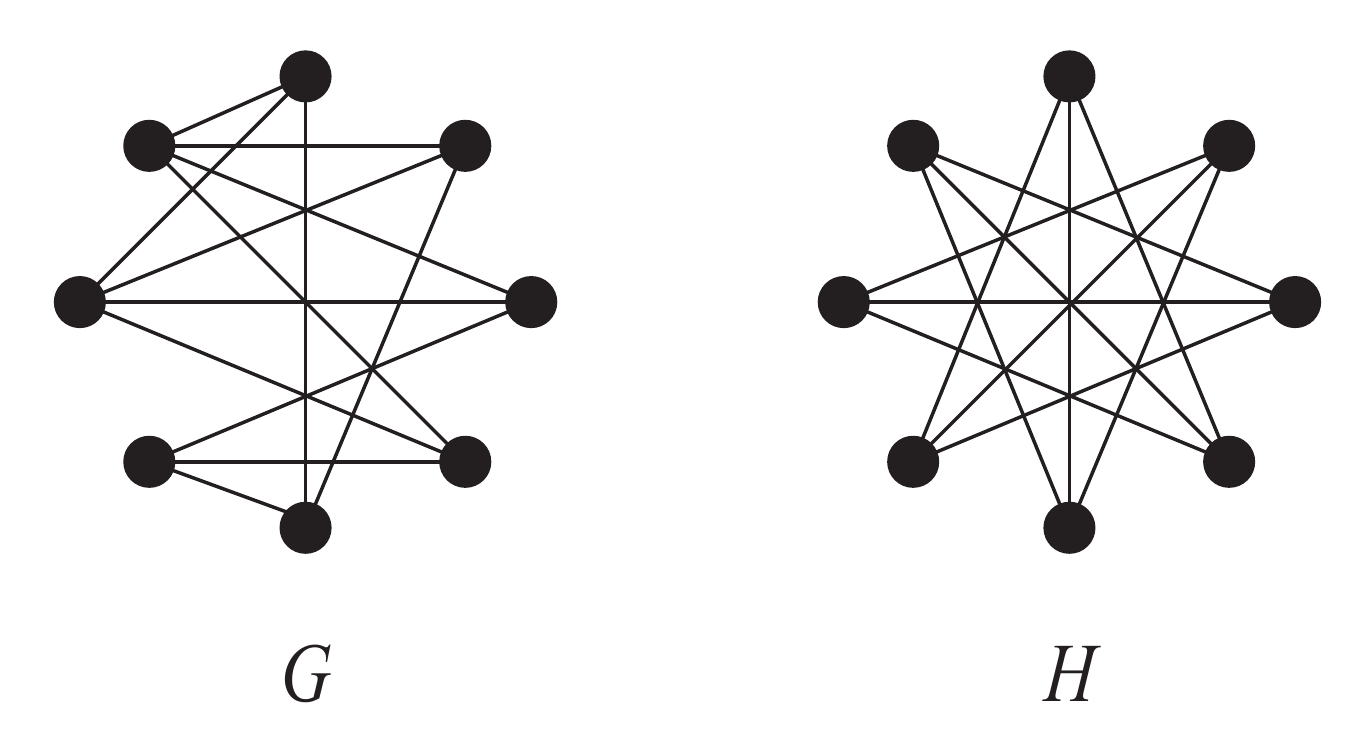}
\caption{}
\label{fig}
\end{center}
\end{figure}

\begin{conjecture}\label{newconjecture} Let $G$ be a $2$-connected $k$-chromatic graph of order $n>k\geq 4$. Then for all $x \geq k$,
\[ \pi(G,x)\leq \frac{ (x)_{\downarrow k}\pi(C_{n-k+2},x)}{x(x-1)},\]
with equality holding if $G$ arises by attaching an ear to $K_{k}$ (an {\em ear} is a new path or cycle that overlaps an existing graph only in its two endpoints).
\end{conjecture}

What about for even higher connectivity? We have found that among all $3$-connected $3$-chromatic graphs of order $8$, the graph $G$ shown at the left of Figure~\ref{fig} is the unique $3$-connected $3$-chromatic graph of order $8$ with the largest number of $3$-colourings ($66$), but the graph $H$ on the right (which happens to be a circulant graph) has the most $4$-colourings, $2140$ (compared to $G$'s $2060$ $4$-colourings). Of course, for any positive integers $l$ and $k$, there is always an $l$-connected $k$-chromatic graph of order $n$ with the most $x$-colourings, provided $x$ is large enough, but our example shows that for some classes, we cannot start necessarily at $x = k$. 

In another direction, it is straightforward to see that if $a_{j}(H) \leq a_{j}(G)$ for all $j$, then $\pi(G,x) \leq \pi(H,x)$ for all $x \geq \chi(H)$. Thus if some graph in a subclass of $k$-chromatic graphs has the largest $a_{j}$ sequence (term-wise) among all such graphs, it necessarily has the largest number of $x$-colourings in the subclass. It seems reasonable that the extremal graphs in $\kchromfamilystar$ have the largest $\langle a_{i} \rangle$ sequence, and likewise for the graphs in Conjecture~\ref{newconjecture}. If we try to extend to $3$-connected graphs, there are not necessarily largest $\langle a_{i} \rangle$ sequences; as mentioned above, the graph $G$ shown at the left of Figure~\ref{fig} is the unique $3$-connected $3$-chromatic graph of order $8$ with the largest number of $3$-colourings among all $3$-connected $3$-chromatic graphs of order $8$, and its $a_{j}$ sequence, $\langle 11,74,124,71,15,1 \rangle$, is thus the only candidate for a largest such sequence, but the graph $H$ on the right has sequence $\langle 8,82,144,60,16,1 \rangle$, so no optimal sequence exists.

\vskip0.4in
\noindent {\bf \large Acknowledgments:} The authors would like to acknowledge the support of the Natural Sciences and Engineering Research Council of Canada. As well, the authors would like to thank Gordon Royle for providing us with files of small graphs with fixed chromatic numbers.

\bibliographystyle{elsarticle-num}

\end{document}